\documentclass[12pt]{amsart}
\usepackage{amsmath,amsthm,epsfig,amssymb,amscd,amsfonts} 
\usepackage[all]{xy}
\usepackage{xcolor}
\usepackage{young}
\newtheorem{thm}[subsection]{Theorem}

\newtheorem{prop}[subsection]{Proposition}
\newtheorem{conjecture}[subsection]{Conjecture}

\newtheorem{lem}[subsection]{Lemma}

\theoremstyle{definition}
\newtheorem{Def}[subsection]{Definition}

\newtheorem{proposition-definition}[subsection]{Proposition-Definition}

\begin{normalsize} \end{normalsize}

\newcommand{\SSS}{{\mathcal S}}

\numberwithin{equation}{section}

\author{F. Laytimi}

\address {F. L.: Math\'ematiques - b\^{a}t. M2, Universit\'e Lille 1,
F-59655 Villeneuve d'Ascq Cedex, France}

\email {fatima.laytimi@math.univ-lille1.fr}

\author{Nagaraj D S}

\address{Indian Institute of Science Education and Research, Tiupati}

\email{dsn@labs.iisertirupati.ac.in}


\title{Towards  Bigness equivalence }

\begin{document}

\date{}

\begin{abstract} 
	
	\bigskip
	On  the flag  variety $ \mathcal{F}l_s(E)$
	associated to a vector bundle $E,$ , a sequence $s$ and a partition $a,$ there is a line bundle $\it Q^a_s$  
        on $ \mathcal{F}l_s(E).$
        
The aim of this paper is to prove:

\begin{conjecture}   $Q^a_s $ on $ \mathcal{F}l_s(E)$ is big if only if $\pi_*(Q^a_s)=S_a(E)$ on X is big.
\end{conjecture}

The "if" part is proven here,  the "only if" part is proven under the V-bigness  hypothesis.

\end{abstract}

\maketitle

\section{Introduction} \setcounter{page}{1}
Let  $E$ be  a complex vector bundle  of rank $e$  on a projective manifold $X$
and $s$ be a sequence of integers $s=(s_0,s_1,\ldots s_m)$ such that  $$0=s_0<s_1<\ldots<s_m\leq e.$$
  Associated to $E$ and $s$ we have the flag  a
  
manifold  
$$ \pi : \mathcal{F}l_s(E) \to X. $$
For $x \in X$ the fiber $\pi^{-1}(x)$
is the incomplete flag variety given by nested subspaces $V_{s_i}$ of codimension $s_i$ satisfying:  
$$V=V_{s_0}\supset V_{s_1}\supset\ \ldots \supset V_{s_m}=\{0\}$$
of the 
fiber $V=E_x$ of $E.$

 Let  $a$  be a partition such that 
$a_{{s_{{j-1}}+1}}=\dots =a_{s_j}$,  $1\leq j\leq m$, we set 
$$ \it Q_s^a= \it Q^{a_{s_1}}_{s_1}\otimes \it Q^{a_{s_2}}_{s_2} \otimes \ldots \otimes \it Q^{a_{s_m}}_{s_m},$$ 
where $$ \it Q_{s_j}=\det(V_{s_{j-1}}/V_{s_{j}}), \ \ 1\leq j\leq m.$$ 

\bigskip

When restricted to a fiber of $\pi,$ the line bundle $\it Q_s^a$ is very ample.

\bigskip

 Note that  [see \cite{Dem}], for any partion $a$ 
\begin{equation}\label{eq1} \pi_*(\it Q_{s}^a) = \SSS_{a}(E) \\
\hskip1cm R^q\pi_*(\it Q_s^a) = 0\ \ \mathit{for}\ \ q>0,\hskip2cm 
\end{equation}
\label{pi}

\smallskip

In this paper we prove:

\begin{thm}\label{main1}
	If the line bundle $\it Q_s^a$ on  $\mathcal{F}l_s(E)$ is big, 
	then  the vector bundle 
	$\pi_*(\it Q_s^a)=\SSS_{a}E$ on $X$  is big.
\end{thm}

\begin{thm}\label{main2}
	Let $E$ be a vector bundle on $X.$ Assume $\SSS_a(E)$ is 
		V-big on $X.$ Then the line bundle $\it Q_s^a$ on  $\mathcal{F}l_s(E)$ is big.
\end{thm}

\section{Proof of theorem 1.1}

For the proof of the Theorem we recall \cite{BLNN} the following two results:

\begin{prop}\label{prop1} Let $F$ be a big and nef vector bundle and $B$ be any line bundle on $X$ then there is an integer $r \geq 1$ such that
	$$ {\rm H}^{0}(S^r(F)\otimes B)\neq (0),$$ 
	where $S^r(F)$ is the $r$-th symmetric power of $F.$
\end{prop}

\begin{proof}
	See Proposition 3.3. \cite{BLNN}.
\end{proof} 

\begin{prop}\label{prop2} Let $F$ be a vector bundle and $A$ be an ample line bundle  on $X$ then $F$ is big if and only if 
	$$ {\rm H}^{0}(S^r(F)\otimes A^{-1})\neq (0)$$ 
	for some $r \geq 1,$
	where  $A^{-1}$
	is the dual of the bundle $A.$
\end{prop}

\begin{proof} See Proposition 3.4. \cite{BLNN}.
\end{proof}

  Let $A$ be an ample bundle on $X$
By Proposition(\ref{prop1}) the assumption that the line bundle $\it Q_s^a$ on  $\mathcal{F}l_s(E)$ is big implies there is an integer $n \geq 1$ such that 
$$ {\rm H}^{0}((\it Q_s^a)^n\otimes \pi^*(A^{-1}))\neq (0)$$
  This implies from \ref{eq1} that 
$$ {\rm H}^{0}((\SSS_{na}(E)))\otimes A^{-1})\neq (0).$$
On the other hand $(\SSS_{na}(E))\otimes A^{-1}$ is a direct summand 
of $S^n(\SSS_{a}(E))\otimes A^{-1}.$ Hence one has 
$$ {\rm H}^{0}((S^n(\SSS_{a}(E)))\otimes A^{-1})\neq (0).$$
Hence by Proposition(\ref{prop2}) we conclude that $\SSS_{a}(E)$ is big.

\section{Proof of theorem 1.2}
Before giving the proof,
we  recall the definition of V-big vector bundle (see \cite{BKKMSU}):

\begin{Def} Let $G$ be a vector bundle and $A$ be an ample line bundle on a projective variety $X.$ The bundle $G$ is said to be V-big,  if 
for some positive integers $m$ and $r_0$ the  bundle $$S^r(S^m(G))\otimes A^{-r}$$
is generated  by its global sections on a non-empty open set $U$ of  $X$ for all $r \geq r_0.$
\end{Def}

Now since $\it Q_s^a$ is $\pi$-ample, where  $$ \pi : \mathcal{F}l_s(E) \to X $$ is the 
natural projection,  for any ample line bundle 
	$A$ on $X$ the bundle $\it Q_s^a\otimes \pi^*(A^r)$ is ample for all large enough integer $r.$
Let $\SSS_a(E)$ be V-big on $X$ and $A$ be an ample line bundle on $X.$  
Then by definition of V-big there exists an integers $m > 0$ and $r_0$ such that for all $r \geq r_0$ the vector bundle bundle 
$$S^r(S^m(\SSS_a(E)))\otimes A^{-r}$$
is generated by its global sections on a non-empty open set $U $ of  $X.$ 

But we know from representation theory 
that 
$$S^m(\SSS_a(E)) = \SSS_{ma}(E)\oplus R_1$$
 and 
$$S^r(S^m(\SSS_a(E))) = \SSS_{rma}(E)\oplus R_2$$
for some vector bundles $R_1, R_2$ on $X.$
On the other hand if $F, F_1, F_2$ are vector bundles on $X$
such that $F = F_1 \oplus F_2,$ then $F$ is generated on a non-empty open set $V$ of  $X$ by global sections of $F$ if and only if 
$F_1$ is generated on $V$ by global sections of $F_1$ and
$F_2$ is generated  on $V$ by global sections of $F_2.$
Thus we conclude that the bundle 
$$\SSS_{rma}(E)\otimes A^{-r}$$
is generated its global sections on a non-empty open set $U$ of $X,$  i.e., the natural map
$${\rm H}^0(X, \SSS_{rma}(E)\otimes A^{-r}) \to (\SSS_{rma}(E)\otimes A^{-r})_x$$
is surjective for all $x \in U.$ Hence on $\mathbb{P}(\SSS_{rma}(E))$ 
the natural map 
$${\rm H}^0(\mathbb{P}(\SSS_{rma}(E)), \mathcal{O}_{\mathbb{P}(\SSS_{rma}(E))}(1)\otimes \pi_{1}^{*}(A^{-r})) \to \mathcal{O}_{\mathbb{P}(\SSS_{rma}(E))}(1)\otimes \pi_{1}^{*}(A^{-r})_y$$
is surjective for all $y \in \pi_1^{-1}(U),$ where 
$$\pi_1: \mathbb{P}(\SSS_{rma}(E)) \to X $$
is the natural projection. We also have an inclusion 
$$\mathcal {F}l_s(E)\subset \mathbb{P}(S_{rma}(E))$$
such that $\mathcal{O}_{\mathbb{P}(S_{rma}(E))}(1)$
restricted to $\mathcal {F}l_s(E)$ is equal to $(\it Q_s^a)^{\otimes rm}$ and $\pi_1 = \pi$ on $\mathcal {F}l_s(E).$ Also, by (\ref{eq1})
we can deduce that for any integer $n\geq 1$
$${\rm H}^0(\mathcal {F}l_s(E), (\it Q_s^a)^{\otimes n}\otimes \pi^{*}(B)) = {\rm H}^0(X, \SSS_{na}(E)\otimes B)$$
for any line bundle $B$ on $X.$Thus we conclude 
$${\rm H}^0(\mathcal {F}l_s(E), (\it Q_s^a)^{\otimes rm}\otimes \pi^{*}(A^{-r})) \to (\it Q_s^a)^{\otimes rm}\otimes \pi^{*}(A^{-r})_y $$
is surjective for all $y \in \pi^{-1}(U).$ In particular 
$${\rm H}^0(\mathcal {F}l_s(E), (\it Q_s^a)^{\otimes rm}\otimes \pi^{*}(A^{-r})) \neq (0).$$
Since
$$(\it Q_s^a)^{\otimes rm}\otimes \pi^{*}(A^{-r})
= (\it Q_s^a)^{\otimes (rm+1)}\otimes (\it Q_s^a)^{-1} \otimes \pi^{*}(A^{-r})$$
and 
 $(\it Q_s^a)\otimes \pi^{*}(A^{r})$ is ample for large $r,$
we see that the line bundle $(\it Q_s^a)^{\otimes (rm+1)}$ is equal to tensor product of an ample
bundle with an effective line bundle and hence the line bundle $(\it Q_s^a)$
is big.
\bigskip

If $A$ is an ample line bundle on $X$ and $M$ is any line bundle  then $A^m\otimes M$ is ample for all large $m.$ \\

This property is also holds for "big" :

\begin{lem}\label{lem1}
Let $L$ be a big line bundle and $M$ be any line bundle on $X.$ Then for any large enough positve integer $N$ the bundle $L^N \otimes M$ is big. 
\end{lem}

\begin{proof}
 note that $L$ is big implies for any ample line bundle $A$ on $X$  there is an positive integer $n$ and an effective line bundle $E$ on $X$ such that $L^n = E \otimes A.$ 
Now choose a positive  integer $m$ such that $A_1 = A^m\otimes M$ is ample. Hence 
$$X\otimes L^{mn} = E^m (\otimes A^m \otimes M) = E^m \otimes A_1.$$
Since $E^m$ is effective hence from  Proposition(\ref{prop2}) we conclude that $X\otimes L^{mn}$ is big.
\end{proof}

In view of the  Lemma (\ref{lem1}) we can ask the following

\bigskip
{\bf Question}: Let $L$ be a big and nef line bundle on a projective variety $X$ and $M.$ Is the line bundle $L^n \otimes M$ nef and big for some $n$ large? 

\bigskip

Unfortunately the answer in general is  negative as shown by the following example

\bigskip
{\bf Example:} Let $C$ be a non singular  projective curve of genus $g \geq 1$ over  $\mathbb{C}$  and $L$ be a degree zero non torsion line bundle on $C.$
For  $P \in C$ consider  the vector bundle $E = L \oplus {\mathcal O}_C (P)$ and set $X = \mathbb{P}(E).$ 
On $X$ the line bundle $\mathcal{O}_{\mathbb{P}(E)}(1)$ 
is big and nef (see, Lemma 2.3.2. \cite{Laz}). But $p^*(\mathcal{O}_C(-P))\otimes \mathcal{O}_{\mathbb{P}(E)}(n) $ is not nef for any $n,$ where 
$p : X \to C$ is the natural projection.

\bigskip

{\bf Acknowledgments}: ``This work was supported in part by the Labex CEMPI (ANR-11-LABX-0007-01)’’.

  \end{document}